\documentclass[11pt]{article}

\usepackage[centering,margin=1in,a4paper]{geometry}

\usepackage{tikz}
\usetikzlibrary{arrows.meta}

\usepackage{datetime}

\usepackage[all]{xy}

\usepackage{amssymb, amsmath, amsthm} 

\usepackage{mathptmx}
\usepackage[scaled=.90]{helvet}
\usepackage{courier}

\usepackage[pdftex]{hyperref}

\hypersetup{
 colorlinks = true,
 linkcolor = black,
 urlcolor = blue,
 citecolor = black
}

\newcommand{\bbC}{\mathbb{C}}
\newcommand{\bbF}{\mathbb{F}}
\newcommand{\bbH}{\mathbb{H}}
\newcommand{\bbL}{\mathbb{L}}

\newcommand{\bbQ}{\mathbb{Q}}
\newcommand{\bbR}{\mathbb{R}}
\newcommand{\bbZ}{\mathbb{Z}}

\newcommand{\rmB}{\mathrm{B}}
\newcommand{\rmC}{\mathrm{C}}

\newcommand{\rmE}{\mathrm{E}}

\newcommand{\rmH}{\mathrm{H}}
\newcommand{\rmK}{\mathrm{K}}

\newcommand{\rmP}{\mathrm{P}}

\newcommand{\rmS}{\mathrm{S}}

\newcommand{\rmZ}{\mathrm{Z}}

\newcommand{\rmd}{\mathrm{d}}

\renewcommand{\int}{\mathrm{int}}

\DeclareMathOperator{\hofib}{hofib}

\DeclareMathOperator{\Gr}{Gr}

\DeclareMathOperator{\Tor}{Tor}
\DeclareMathOperator{\Hom}{Hom}

\DeclareMathOperator{\Ker}{Ker}

\DeclareMathOperator{\Cell}{Cell}

\newtheorem{theorem}{Theorem}
\newtheorem{proposition}[theorem]{Proposition}
\newtheorem{corollary}[theorem]{Corollary}

\theoremstyle{definition}
\newtheorem{definition}[theorem]{Definition}
\newtheorem{example}[theorem]{Example}
\newtheorem{examples}[theorem]{Examples}
\newtheorem{remark}[theorem]{Remark}
\newtheorem{notation}[theorem]{Notation}
\newtheorem{convention}[theorem]{Convention}

\numberwithin{theorem}{section}
\numberwithin{equation}{section}
\numberwithin{figure}{section}

\title{\bf Preludes to the Eilenberg--Moore and the Leray--Serre spectral sequences}

\author{Frank Neumann and Markus Szymik}

\newdateformat{mydate}{\monthname~\twodigit{\THEYEAR}}
\date{\mydate\today}

\begin{document}

\maketitle

%%%

The Leray--Serre and the Eilenberg--Moore spectral sequences are fundamental tools for computing the cohomology of a group or, more generally, of a space. We describe the relationship between these two spectral sequences when both of them share the same abutment. There exists a joint tri-graded refinement of the Leray--Serre and the Eilenberg--Moore spectral sequence. This refinement involves two more spectral sequences, the preludes from the title, which abut to the initial terms of the Leray--Serre and the Eilenberg--Moore spectral sequence, respectively. We show that one of these always degenerates from its second page on and that the other one satisfies a local-to-global property: It degenerates for all possible base spaces if and only if it does so when the base space is contractible. 

%%%

\section{Introduction}

%%%

We consider principal fibration sequences~$\Omega Z\to X\to Y\to Z$ for connected spaces~$X$,~$Y$, and~$Z$, with the latter also simply-connected or at least nilpotent. For instance, these could be classifying spaces for a central extension of discrete groups, with~$Z\simeq\rmK(A,2)$ an Eilenberg--Mac Lane space for the abelian kernel~$A$. In general, there are two spectral sequences that converge to the cohomology of~$X$ with coefficients in a field~$K$, say: a Leray--Serre spectral sequence for~\hbox{$\Omega Z\to X\to Y$} in the first quadrant and an Eilenberg--Moore spectral sequence for~\hbox{$X\to Y\to Z$} in the second. 
It has long been recognised that these two spectral sequences often process the same information in different ways~(e.g., see~\cite{Leary}).

The main goal of this paper is to replace intuition with certainty. We provide systematic tools and examples showing that our assumptions are reasonable. It turns out that a comparison{\it~is} always possible, and that, in favorable cases, it leads to computational effects. Our main point is the following result.

%%%

\begin{theorem}\label{thm:existence}
For every principal fibration sequence~$\Omega Z\to X\to Y\to Z$ there are two spectral sequences starting with the same tri-graded groups~$\rmE^{s,t,u}_1$, involving different pairs~$(s, u)$ and~$(t, u)$ of indices and converging to the~$\rmE_1$ pages of the Leray--Serre spectral sequence for~\hbox{$\Omega Z\to X\to Y$} and the Eilenberg--Moore spectral sequence for~\hbox{$X\to Y\to Z$}, respectively, in turn abutting both to~$\rmH^{s+t+u}(X)$.
\end{theorem}

These two spectral sequences, which are introduced at the beginning of Section~\ref{sec:comparison}, are the{\it~preludes} from the title. Existence results for spectral sequences are nowadays mostly a formality, and Theorem~\ref{thm:existence} follows a pattern introduced by Deligne~\cite{Deligne:II}~(see also Miller~\cite{Miller}). The crucial point of the above result is the agreement of the initial terms of the preludes, and its main thrust comes from our two accompanying degeneracy results. For the Eilenberg--Moore spectral sequence, we have the following:

\begin{theorem}
The preludes to the Eilenberg--Moore spectral sequences always degenerate from their~$\rmE_2$ page on. They degenerate from their~$\rmE_1$ page on if and only if the space~$Y$ is~$K$--minimal.
\end{theorem}

This is proven as Theorem~\ref{thm:LS_is_short} below. It requires the notion of{\it~$K$--minimality}, which refers to the existence of a cellular structure such that the cellular chain complex with coefficients in the field~$K$ is minimal in the sense that the differential is trivial~(see Definition~\ref{def:minimal}). Needless to say, many important classes of spaces have this property~(see Examples~\ref{ex:good_spaces_1} for a start). The counterpart of Theorem~\ref{thm:LS_is_short} for the Leray--Serre spectral sequence is Theorem~\ref{thm:EM_is_local}, a local-to-global principle for degeneracy:

\begin{theorem}
Given a space~$Z$, the preludes to the Leray--Serre spectral sequences degenerate for all~$Y$ from their~$\rmE_2$ page on if and only if this holds for a single point~\hbox{$Y=\star$}.
\end{theorem}

We will say that a space~$Z$ to which the result applies is{\it~$K$--unbarred}~(see Definition~\ref{def:unbarred}). Once again, many important classes of spaces have this property~(see Examples~\ref{ex:good_spaces_2}). In particular, spaces with polynomial cohomology~(Proposition~\ref{prop:polynomial_implies_unbarred}) and suspensions~(Proposition~\ref{prop:suspensions_are_unbarred}) are unbarred.

%%%

An abundance of examples demonstrates that neither the Eilenberg--Moore nor the Leray--Serre spectral sequence is more efficient (i.e., closer to the abutment) than the other. Sometimes, one is broken in half, with either algebraic or geometric differentials prepended in one of the preludes. Still, in general, both of them are necessary, and our results suggest always considering the whole quartet of spectral sequences featured in Theorem~\ref{thm:existence}. 

%%%

Here is an outline of this article. In the following Section~\ref{sec:ss}, we briefly review the construction of the spectral sequence of a filtered complex and how the Leray--Serre and the Eilenberg--Moore spectral sequences are particular cases of that. We also discuss the classes of~$K$--minimal and~$K$--unbarred spaces, the two conditions that appear naturally as assumptions in the following Section~\ref{sec:comparison}, where we prove our main comparison results. The final Section~\ref{sec:examples} includes some elementary but illustrative examples of degenerate cases and the principal fibrations in which the Hopf fibration features. Its purpose is purely didactic; more substantial applications will appear elsewhere.

\begin{convention}
We are working with a field~$K$ of coefficients throughout.
\end{convention}

%%%

\section{Spectral sequences}\label{sec:ss}

In this section, we provide the necessary definitions and notations used later on. 
We start with the briefest review of the spectral sequence defined by a cochain complex that comes with a decreasing filtration. 
We also explain how Leray--Serre and Eilenberg--Moore spectral sequences are both instances of this general construction. The reader may want to skip this section and refer back to it when needed, but we point out the two non-standard Definitions~\ref{def:minimal} and~\ref{def:unbarred}. As a standard reference for spectral sequences, we refer to McCleary's textbook~\cite{McCleary}. 

%%%

\subsection{The spectral sequence of a filtered cochain complex}

We consider cochain complexes~$C$ with differentials~$\rmd\colon C^n\to C^{n+1}$ that increase the cohomological degree. Occasionally, we will have to use the shifted complex~$C[d]$ that satisfies~$C[d]^n=C^{d+n}$. 

When given a descending filtration~$F$ on such a cochain complex~$C$ by subcomplexes~\hbox{$F^pC\geqslant F^{p+1}C$}, the associated graded is defined as~$\Gr_F^pC=F^pC/F^{p+1}C$. We set
\begin{equation}\label{eq:F}
\rmZ^{p,q}_r(C,F)=\{\,x\in F^pC^{p+q}\,|\,\rmd x\in F^{p+r}C^{p+q+1}\,\},
\end{equation}
and this leads to a spectral sequence~$(\rmE_r,\rmd_r)$ which starts with
\[
\rmE^{p,q}_0(C,F)=\Gr_F^pC^{p+q},
\]
and the differential is of the form~\hbox{$\rmd_r\colon\rmE^{p,q}_r(C,F)\to\rmE^{p+r,q-r+1}_r(C,F)$} induced by the differential~$\rmd$ of the original cochain complex~$C$. We have
\[
\rmE^{p,q}_1(C,F)=\rmH^{p+q}(\Gr_F^pC).
\]
The spectral sequence abuts to the associated graded of the corresponding filtration on the cohomology~$\rmH^\bullet(C)$. Suitable references for the theory are Godement~\cite{Godement}, Grothendieck~\cite[III.0~\S11]{EGA:III}, and Deligne~\cite{Deligne:II}.

%%%

\subsection{Filtrations on spaces and minimality}\label{sec:cellular}

Let~$X$ be a space with an increasing filtration by subspaces~$X^s\leqslant X^{s+1}$. Any such increasing filtration on~$X$ gives a decreasing filtration of the cochain complex~$\rmC(X)$ by setting
\begin{equation}\label{eq:F-filtration}
F^s\rmC(X)=\Ker(\rmC(X)\to\rmC(X^{s-1}))=\rmC(X,X^{s-1}),
\end{equation}
and this results in a spectral sequence starting from
\[
\rmE^{s,t}_0=\Gr_F^s\rmC^{s+t}(X)=\frac{\rmC(X,X^{s-1})^{s+t}}{\rmC(X,X^s)^{s+t}}=\rmC(X^s,X^{s-1})^{s+t}. 
\]
Passing to cohomology, we get
\begin{equation}\label{eq:Leray--SerreE1}
\rmE^{s,t}_1=\rmH^{s+t}(X^s,X^{s-1})\Longrightarrow\rmH^{s+t}(X).
\end{equation}

The following turns out to be a conceptually important special case.

\begin{example}\label{ex:cell}
Let $X$ be a CW-complex endowed with the skeletal filtration. We shall write
\[
\Cell^s(X;V)
=\Hom(\rmH_\bullet(X^s,X^{s-1}),V)
=\rmH^\bullet(X^s,X^{s-1})\otimes V
\]
for the cellular cochains of~$X$ with coefficients in a~$K$--vector space~$V$. The~$\rmE_1$ page of the spectral sequence~\eqref{eq:Leray--SerreE1} is concentrated in the row with~$t=0$, and
\[
\rmE_1^{s,0}=\Cell^s(X;K)
\]
is the vector space of~$K$--valued functions on the~$s$--cells of~$X$. The~$\rmd_1$ differential turns the groups above into the cellular cochain complex of~$X$ with coefficients in the field~$K$, and from its~$\rmE_2$ page on, this spectral sequence necessarily degenerates.
\end{example}

We often need to refer to the situation when a spectral sequence as in Example~\ref{ex:cell} degenerates from its~$\rmE_1$ page on. There does not seem to be established terminology for the following.

\begin{definition}\label{def:minimal}
Given a field~$K$, a space is called{\it~$K$--minimal} if there exists a~CW-structure on it such that the cellular chain complex with respect to this~CW structure and coefficients in the field~$K$ has trivial differential.
\end{definition}

\begin{examples}\label{ex:good_spaces_1}
The circle~$\rmS^1\simeq \rmB\bbZ$, and more generally the spheres~$\rmS^r$ and tori~$(\rmS^1)^r$ are~$K$--minimal for any field~$K$. The real projective space~$\bbR\rmP^\infty\simeq\rmB\bbZ/2$ is~$K$--minimal with respect to the field~$\bbF_2$, but {\it not} with respect to the field~$\bbQ$.
\end{examples}

\begin{remark}
Spaces with polynomial~$K$--cohomology are {\it not} automatically~$K$--minimal. For instance, any~$K$--acyclic space automatically has polynomial~$K$--cohomology for trivial reasons, but it is~$K$--minimal if and only if it is contractible.
\end{remark}

%%%

\subsection{Leray--Serre spectral sequences}\label{sec:LSSS}

The filtrations on spaces~$X$ that we mostly care about come, more generally, from maps~$X\to Y$ to another space~$Y$. These maps allow us to pull back a skeletal filtration of~$Y$, for instance. Then~\eqref{eq:Leray--SerreE1} is what we will refer to as the {\it Leray--Serre spectral sequence} for the map~$X\to Y$ with respect to the given filtration on~$Y$. The~$\rmE_1$ page of a Leray--Serre spectral sequence is particularly easy to describe when the map~$X\to Y$ is a fibration with a simply-connected base~$Y$, in terms of the fibre. More generally, if the map in question might not be a fibration, we have to use the homotopy fibre, and then the~$\rmE_2$ page is
\[
\rmE^{s,t}_2\cong\rmH^s(Y;\rmH^t(\hofib(X\to Y)))
\]
This does {\it not} depend on the filtration on~$Y$ any longer. In our standard situation
\begin{equation}\label{eq:XYZ}
\Omega Z\longrightarrow X\longrightarrow Y\longrightarrow Z,
\end{equation}
the homotopy fibre of the map~$X\to Y$ is a loop space~$\Omega Z$, and the Leray--Serre spectral sequence looks as follows.
\begin{equation}\label{eq:LSSS}
\rmE^{s,t}_2\cong\rmH^s(Y;\rmH^t(\Omega Z))\Longrightarrow\rmH^{s+t}(X)
\end{equation}
We refer to Remark~\ref{rem:nil_2} for an argument showing that this description is valid in our specific situation~\eqref{eq:XYZ}, even if~$Y$ is not simply-connected.

%%%

\subsection{Eilenberg--Moore spectral sequences}\label{sec:EMSStop}

Let~$A$ be an augmented differential graded~$K$--algebra, for the ground field~$K$, with augmentation ideal~$\overline{A}$. Given a differential graded right~$A$--module~$M$ and a differential graded left~$A$--module~$N$, let
\begin{equation}\label{eq:bar_complex}
\rmB(M,A,N)
\end{equation}
denote the {\it (normalized) bar complex}~\cite[Ch.~II]{Eilenberg+MacLane} of Eilenberg and Mac Lane. In our conventions, the object~$\rmB(M,A,N)$ is a cochain complex of~$K$--vector spaces, and its homogeneous elements are expressions
\[
m\,[\,a_1\,|\dots|\,a_k\,]\,n.
\]
The grading is such that this element sits in degree
\begin{equation}\label{eq:deg_shift}
\deg(m\,[\,a_1\,|\dots|\,a_k\,]\,n)=\deg(m)+\sum_j\deg(a_j)+\deg(n)-k.
\end{equation}
The differential on the bar complex is the sum of an internal and an external differential. The complex~$\rmB(M,A,N)$ serves as a specific model for the derived tensor product, the differential graded~$K$--vector space~\hbox{$M\otimes_A^{\bbL}N$}.

The bar complex comes with a descending filtration~$W$ defined by
\begin{equation}\label{eq:W_filtration}
W^p\rmB(M,A,N)=\mathrm{Span}\{\,m\,[\,a_1\,|\dots|\,a_k\,]\,n\,|\,k\leqslant -p\,\},
\end{equation}
so that~$W^{-\infty}\rmB(M,A,N)=\rmB(M,A,N)$ and~$W^1\rmB(M,A,N)=0$. This filtration of the bar complex leads to a spectral sequence with
\begin{equation}\label{eq:E_0}
\rmE^{p,q}_0=\Gr^p_W\rmB(M,A,N)^{p+q}=(M\otimes\overline{A}^{\otimes-p}\otimes N)^q,
\end{equation}
as follows from~\eqref{eq:deg_shift}. The abutment is~$\rmH^{p+q}(\rmB(M,A,N))$. The~$\rmd_0$  differential is the internal one, up to sign, so that
\begin{equation}\label{eq:E_1}
\rmE^{p,q}_1=(\rmH^\bullet(M)\otimes\overline{\rmH}^\bullet(A)^{\otimes-p}\otimes\rmH^\bullet(N))^q
\end{equation}
by the K\"unneth theorem. Together with the~$\rmd_1$ differential, which is induced from the external one, the~$\rmE_1$ page is a complex that can be used to compute the graded Tor over the cohomology algebra~$\rmH^\bullet(A)$, and we get
\begin{equation}\label{eq:E_2}
\rmE^{p,q}_2=\Tor_{-p}^{\rmH^\bullet(A)}(\rmH^\bullet(M),\rmH^\bullet(N))^q\Longrightarrow\rmH^{p+q}(\rmB(M,A,N)).
\end{equation}
We will refer to this spectral sequence as the {\it bar spectral sequence}. It is also referred to as the {\it algebraic Eilenberg--Moore spectral sequence}~\cite{Eilenberg+Moore,Gugenheim+May,Halperin+Stasheff} and sometimes as the {\it Moore spectral sequence}~\cite{Moore, Clark}.

Let~$X=P\times_ZY$ be a homotopy pullback of spaces~(with~$Z$ nilpotent). If we again write~$\rmC(X)$ for the cochains on~$X$, and similarly for~$P$,~$Z$, and~$Y$, then there is an equivalence
\begin{equation}\label{eq:EMequivalence}
\rmC(X)\simeq\rmB(\rmC(P),\rmC(Z),\rmC(Y))
\end{equation}
of cochain complexes~\cite{Eilenberg+Moore}~(see also~\cite{Smith:TAMS}). The right-hand side is the bar complex~\eqref{eq:bar_complex}. The equivalence becomes particularly evident when~$X$ itself is described as the totalization of a geometric bar construction~\cite{Rector}. 

We shall later need a relative version of the Eilenberg--Moore equivalence~\eqref{eq:EMequivalence}:

\begin{proposition}\label{prop:relativeEM}
Assume that~$P$,~$X$,~$Y$, and~$Z$ are as above. If~$B\subseteq Y$ is a subspace, and~$A\subseteq X$ the restriction to~$X$, then the bar complex~$\rmB(\rmC(P),\rmC(Z),\rmC(Y,B))$, with its total differential, is equivalent to~$\rmC(X,A)$. 
\end{proposition}

\begin{proof}
We have a short exact sequence
\begin{equation}\label{eq:(Y,B)}
0
\longrightarrow\rmC(Y,B)
\longrightarrow\rmC(Y)
\longrightarrow\rmC(B)
\longrightarrow0
\end{equation}
of cochain complexes, and all terms are canonically~$\rmC(Y)$--modules and therefore~$\rmC(Z)$--modules by restriction along~\hbox{$\rmC(Z)\to\rmC(Y)$} induced by~\hbox{$Y\to Z$}. Application of the functor~${?}\mapsto\rmB(\rmC(P),\rmC(Z),{?})$ leads to another short exact sequence
\[
0
\to\rmB(\rmC(P),\rmC(Z),\rmC(Y,B))
\to\rmB(\rmC(P),\rmC(Z),\rmC(Y))
\to\rmB(\rmC(P),\rmC(Z),\rmC(B))
\to0.
\]
We compare this short exact sequence with the short exact sequence 
\[
0
\longrightarrow\rmC(X,A)
\longrightarrow\rmC(X)
\longrightarrow\rmC(A)
\longrightarrow0,
\]
which is analogous to~\eqref{eq:(Y,B)}. Two out of three comparison maps are equivalences by the absolute Eilenberg--Moore theorem and the relative one follows.
\end{proof}

%%%

The equivalence~\eqref{eq:EMequivalence} implies that we have a spectral sequence
\[
\rmE^{p,q}_2=\Tor_{-p}^{\rmH^\bullet(Z)}(\rmH^\bullet(P),\rmH^\bullet(Y))^q
\Longrightarrow\rmH^{p+q}(\rmB(\rmC(P),\rmC(Z),\rmC(Y)))\cong\rmH^{p+q}(X).
\]
This is the {\it Eilenberg--Moore spectral sequence}.

The following special cases will be relevant for us:

\begin{example}\label{ex:EMSS_fibres}
If~$P=\star$ is a point, the space~$X$ is the homotopy fibre of the map~\hbox{$Y\to Z$}. We have~\hbox{$\rmH^\bullet(P)=K$}, the ground field, and we get a spectral sequence
\begin{equation}\label{eq:EMSS_fibres}
\rmE^{p,q}_2=\Tor_{-p}^{\rmH^\bullet(Z)}(K,\rmH^\bullet(Y))^q\Longrightarrow\rmH^{p+q}(X).
\end{equation}
This is the usual description of the Eilenberg--Moore spectral sequence for our standard situation~\eqref{eq:XYZ}. A version of this Eilenberg--Moore spectral sequence, for homology with integral coefficients, has been implemented recently using functional programming~\cite{RRSS}.
\end{example}

\begin{example}\label{ex:EMSS_loops}
If, in addition to~$P=\star$, the space~$Y=\star$ is also a point, then~$X\simeq\Omega Z$ is the loop space of~$Z$ and we get a spectral sequence
\begin{equation}\label{eq:EMSS_loops}
\rmE^{p,q}_2=\Tor_{-p}^{\rmH^\bullet(Z)}(K,K)^q\Longrightarrow\rmH^{p+q}(\Omega Z).
\end{equation}
\end{example}

\begin{remark}
There are many interesting situations in which an Eilenberg--Moore spectral sequence does not degenerate from its~$\rmE_2$ page on~(see Example~\ref{ex:EM_diffs} below and~\cite{Schochet, Kraines+Schochet, Rusin:2}). 
\end{remark}

%%%

\subsection{Unbarred spaces}\label{sec:unbarred}

We need to single out a class of spaces that is well-adapted to the situations we are interested in.

\begin{definition}\label{def:unbarred}
Given a field~$K$, we say that a nilpotent space~$Z$ is{\it~$K$--unbarred} with respect to the field~$K$ if the Eilenberg--Moore spectral sequence~\eqref{eq:EMSS_loops} for the loop space~$\Omega Z$ degenerates from its~$\rmE_2$ page on.
\end{definition}

We will see later, in Theorem~\ref{thm:EM_is_local}, that the condition in Definition~\ref{def:unbarred} implies also the degeneration of all other Eilenberg--Moore spectral sequences for all other fibrations with base~$Z$.

\begin{examples}\label{ex:good_spaces_2}
The circle~$\rmS^1\simeq\rmB\bbZ$, and more generally the spheres~$\rmS^r$ and tori~$(\rmS^1)^r$ are~$K$--unbarred for any field~$K$. The real projective space~$\bbR\rmP^\infty\simeq\rmB\bbZ/2$ is ~$\bbF_2$--unbarred, but the analog for odd primes is wrong~(see Example~\ref{ex:EM_diffs}): the classifying space~$\rmB\bbZ/\ell$ is barred for the field~$\bbF_\ell$ if~$\ell$ is an odd prime.
\end{examples}

\begin{proposition}\label{prop:polynomial_implies_unbarred}
Spaces with polynomial~$K$--cohomology are~$K$--unbarred.
\end{proposition}

\begin{proof}
If the cohomology~$\rmH^\bullet(Z)$ of a space~$Z$ is a polynomial algebra over the ground field~$K$, the Koszul complex shows that~$\Tor^{\rmH^\bullet(Z)}_{\bullet,\bullet}(K,K)$ is an exterior algebra. This is the~$\rmE_2$ page of the Eilenberg--Moore spectral sequence. The exterior algebra generators sit in the column~$p=-1$. Therefore, all differentials~$\rmd_r$ for~$r\geqslant2$ vanish on the generators, and the spectral sequence degenerates.
\end{proof}

\begin{proposition}\label{prop:suspensions_are_unbarred}
Suspensions are~$K$--unbarred for any field~$K$.
\end{proposition}

\begin{proof}
For a suspension~$Z=\Sigma Z'$ of a space~$Z'$, we have to show that the Eilenberg--Moore spectral sequence for the loop space fibration~$\Omega\Sigma Z'\to\star\to\Sigma Z'$ degenerates. The cohomology algebra~$\rmH^\bullet(\Sigma Z')$ has trivial multiplication and is Koszul. The Koszul dual
\[
\rmH^\bullet(\Sigma Z')^!=\Tor^{\rmH^\bullet(\Sigma Z')}_{\bullet,\bullet}(K,K)
\]
is the tensor algebra of the reduced cohomology~$\overline{\rmH}^\bullet(\Sigma Z')$, and it is concentrated on the diagonal of the~$\rmE_2$ page. This agrees with the Bott--Samelson theorem: the cohomology of~$\rmH^\bullet(\Omega\Sigma Z')$ is~(additively) the tensor algebra on the reduced cohomology, and this is the~$\rmE_\infty$ page. Therefore, there can be no differentials in the Eilenberg--Moore spectral sequence from~$\rmE_2$ on.
\end{proof}

\begin{remark}
A different proof could be based on arguments in Smith's survey~\cite[Ex.~4 in Sec.~2]{Smith:PSPM}. Those were later generalized by Baker~\cite[Theorem 2.4.]{Baker} from suspensions to certain Thom spaces.
\end{remark}

%%%

\section{Comparison}\label{sec:comparison}

In this section, we first explain a general procedure to compare the two spectral sequences we have when given two filtrations on the same cochain complex~$B$, following Deligne~\cite{Deligne:II}~(see Section~\ref{sec:Zassenhaus}). We then return to the situation of a principal fibration~\eqref{eq:XYZ} and apply this procedure to the cochain complex~\hbox{$B=\rmB(K,\rmC(Z),\rmC(Y))\simeq\rmC(X)$} from the bar construction which supports two filtrations~$F$ and~$W$ that lead to a Leray--Serre spectral sequence~(Section~\ref{sec:idLS}) and an Eilenberg--Moore spectral sequence~(Section~\ref{sec:idEM}), respectively, both with abutment~$\rmH^\bullet(X)$.
As a result, we obtain two more spectral sequences, both supported on the tri-graded vector space
\[
\rmE^{s,t,u}_1
=\rmH^\bullet(Y^s,Y^{s-1})\otimes(\overline{\rmH}^\bullet(Z)^{\otimes-t})^u
=\Cell^s(Y;(\overline{\rmH}^\bullet(Z)^{\otimes-t})^u),
\]
and which abut to the~$\rmE_1$ page of an Eilenberg--Moore spectral sequence~(as we show in Section~\ref{sec:preEM}) and the~$\rmE_1$ page of a Leray--Serre spectral sequence~(as we show in Section~\ref{sec:preLS}), respectively, both for~$\rmH^\bullet(X)$. We then prove that these spectral sequences very often degenerate early, and we spell out the consequences, with Theorems~\ref{thm:LS_is_short} and~\ref{thm:EM_is_local} being the main results.

%%%

\subsection{Zassenhaus squares}\label{sec:Zassenhaus}

Let~$B$ be a cochain complex with two filtrations~$F$ and~$W$. From the two filtrations, we get two spectral sequences with abutment~$\rmH^n(B)$:
\begin{equation}\label{eq:F_spec_seq}
\rmE^{s,n-s}_1(B,F)\cong\rmH^n(\Gr_F^sB)\Longrightarrow\rmH^n(B)
\end{equation}
and
\begin{equation}\label{eq:W_spec_seq}
\rmE^{t,n-t}_1(B,W)\cong\rmH^n(\Gr_W^tB)\Longrightarrow\rmH^n(B).
\end{equation}
The Zassenhaus Lemma gives us canonical isomorphisms
\begin{equation}\label{eq:Zassenhaus_Lemma}
\Gr_F^s\Gr_W^tB\cong\Gr_W^t\Gr_F^sB 
\end{equation}
for all~$s,t\in\bbZ$, and these allow us, following Deligne's thesis~\cite[1.4.9]{Deligne:II}, to relate the two spectral sequences~\eqref{eq:F_spec_seq} and~\eqref{eq:W_spec_seq}. The chain complexes~\eqref{eq:Zassenhaus_Lemma} come equipped with a differential which is induced from~$B$, and their cohomology is the tri-graded vector space
\[
\rmE^{s,t,u}_1\cong\rmH^n(\Gr_F^s\Gr_W^tB),
\]
with
\[
n=s+t+u.
\]
These vector spaces support two~$\rmd_1$ differentials, say~$\rmd_1^F$ and~$\rmd_1^W$, as parts of spectral sequences with abutments 
$\rmH^n(\Gr_W^tB)$ and~$\rmH^n(\Gr_F^sB)$. These abutments are the vector spaces on the~$\rmE_1$ pages of the spectral sequences~\eqref{eq:F_spec_seq} and~\eqref{eq:W_spec_seq}, respectively. It is worth spelling out the indices in detail: the differential~\hbox{$\rmd_1^F\colon\rmE^{s,t,u}_1\to\rmE^{s+1,t,u}_1$} is part of a spectral sequence
\begin{equation}\label{eq:t_spec_seq}
\rmE^{s,t,u}_1\cong\rmH^{s+t+u}(\Gr_F^s\Gr_W^tB)
\Longrightarrow
\rmE^{t,s+u}_1(B,W)\cong\rmH^{s+t+u}(\Gr_W^tB),
\end{equation}
whereas the differential~$\rmd_1^W\colon\rmE^{s,t,u}_1\to\rmE^{s,t+1,u}_1$ is part of a spectral sequence
\begin{equation}\label{eq:s_spec_seq}
\rmE^{s,t,u}_1\cong\rmH^{s+t+u}(\Gr_F^s\Gr_W^tB)
\Longrightarrow
\rmE^{s,t+u}_1(B,F)\cong\rmH^{s+t+u}(\Gr_F^sB).
\end{equation}
We can summarize the present situation concisely in the following {\it Zassenhaus square}~(compare with~\cite[(1.4.9.2)]{Deligne:II} and~\cite[(4.1)]{Miller}).
\[
\xymatrix@!=2em{
& \rmE_1^{s,t,u}\ar@{=>}[dl]_{\eqref{eq:t_spec_seq}}\ar@{=>}[dr]^{\eqref{eq:s_spec_seq}} & \\
\rmE_1^{t,s+u}(B,W)\ar@{=>}[dr]_{\eqref{eq:W_spec_seq}}&&\rmE_1^{s,t+u}(B,F)\ar@{=>}[dl]^{\eqref{eq:F_spec_seq}}\\
&\rmH^{s+t+u}(B)&
}
\]

\begin{notation}\label{not:pq}
Sometimes, the spectral sequences call for indices with~$n=p+q$. We shall use
\[
s=p+n,\quad t=p,\quad u=-2p
\]
on those occasions.
\end{notation}

%%%

The following general result is the blueprint to show that the index transformation is exactly as expected.

\begin{proposition}\label{prop:decalage}
If the spectral sequences~\eqref{eq:t_spec_seq} and~\eqref{eq:s_spec_seq} both degenerate, so that we have isomorphisms
\begin{equation}\label{eq:W_iso}
\rmE_1^{t,n-t}(W)\cong\bigoplus_{s+u=n-t}\rmE^{s,t,u}_1
\end{equation}
and 
\begin{equation}\label{eq:F_iso}
\rmE_1^{s,n-s}(F)\cong\bigoplus_{t+u=n-s}\rmE^{s,t,u}_1,
\end{equation}
then the spectral sequences~\eqref{eq:F_spec_seq} and~\eqref{eq:W_spec_seq} are related by
\begin{equation}\label{eq:as_in_decalage}
\rmE_1^{p,n-p}(W)\cong\rmE_1^{p+n,-p}(F)
\end{equation}
under the index transformation in Notation~\ref{not:pq}.
\end{proposition}

\begin{proof}
Setting~$t=p$ from Notation~\ref{not:pq} into~\eqref{eq:W_iso}, we get
\[
\rmE_1^{p,n-p}(W)\cong
\bigoplus_{s+u=n-p}\rmE^{s,p,u}_1
\]
With~$s=p+n$, this gives
\[
\bigoplus_{s+u=n-p}\rmE^{s,p,u}_1
\cong\bigoplus_{p+n+u=n-p}\rmE^{p+n,p,u}_1
\cong\bigoplus_{p+u=-p}\rmE^{p+n,p,u}_1,
\]
cancelling the~$n$ in the summation index. With~$t=p$ and~$s=p+n$ in~\eqref{eq:F_iso}, we come back to
\[
\bigoplus_{t+u=-p}\rmE^{p+n,t,u}_1\cong\rmE_1^{p+n,-p}(F),
\]
and this finishes the proof.
\end{proof}

%%%

In the rest of this section, we develop the four spectral sequences~\eqref{eq:F_spec_seq}, \eqref{eq:W_spec_seq}, \eqref{eq:t_spec_seq}, and~\eqref{eq:s_spec_seq} for the bifiltered cochain complex~\hbox{$B=\rmB(K,\rmC(Z),\rmC(Y))\simeq\rmC(X)$}, where the two filtrations come from a chosen cellular filtration~$F$ on~$Y$ and the bar filtration~$W$.

%%%

\subsection{Identifying the Leray--Serre spectral sequence}\label{sec:idLS}

We consider the spectral sequence~\eqref{eq:F_spec_seq} for the cochain complex~\hbox{$B=\rmB(K,\rmC(Z),\rmC(Y))\simeq\rmC(X)$}. The filtration~$F$ is induced from a chosen cellular filtration on~$Y$ as in~\eqref{eq:F-filtration}.

\begin{proposition}\label{prop:identifyingLSSS}
For the~$F$--filtered cochain complex~\hbox{$B=\rmB(K,\rmC(Z),\rmC(Y))$}, the spectral sequence~\eqref{eq:F_spec_seq} agrees from its~$\rmE_1$ page on with the Leray--Serre spectral sequence for the principal fibration sequence~\hbox{$\Omega Z\to X\to Y$}.
\end{proposition}

\begin{proof}
This follows by noting that the Eilenberg--Moore equivalence~\eqref{eq:EMequivalence} is natural and, therefore, preserves the~$F$--filtrations on both sides. The induced map between the associated graded cochain complexes is an equivalence 
\[
\rmB(K,\rmC(Z),\rmC(Y^s,Y^{s-1}))\simeq\rmC(X^s,X^{s-1}) 
\]
of cochain complexes by the relative Eilenberg--Moore theorem~(Proposition~\ref{prop:relativeEM} above), and it induces an isomorphism from the next page, their~$\rmE_1$ pages, on.
\end{proof}

%%%

\subsection{Identifying the Eilenberg--Moore spectral sequence}\label{sec:idEM}

We now consider the spectral sequence~\eqref{eq:W_spec_seq} for the cochain complex~\hbox{$B=\rmB(K,\rmC(Z),\rmC(Y))\simeq\rmC(X)$} with the~$W$--filtration~\eqref{eq:W_filtration} on it. 

\begin{proposition}\label{prop:identifyingEMSS}
For the cochain complex~\hbox{$B=\rmB(K,\rmC(Z),\rmC(Y))$} of a space~$X$ with its~$W$--filtration, the spectral sequence~\eqref{eq:W_spec_seq} agrees with the Eilenberg--Moore spectral sequence for the fibration sequence~\hbox{$X\to Y\to Z$}.
\end{proposition}

\begin{proof}
In contrast to Proposition~\ref{prop:identifyingLSSS}, which required some identifications, this is basically the standard construction of the Eilenberg--Moore spectral sequence. The~$\rmE_0$ page is given by~$\Gr_W^t B$, and by~\eqref{eq:E_0} this is isomorphic to
\begin{equation}\label{eq:shift}
\rmE^{t,n-t}_0(B,W)=\Gr_W^t B^n
=(\overline{\rmC}(Z)^{\otimes-t}\otimes\rmC(Y))^{n-t}.
\end{equation}
The~$\rmd_0$ differential on this is induced from~$B$, but on the associated graded for the~$W$--filtration, we only see the internal differential from the cochain complexes~$\rmC({?})$. Therefore, by passing to cohomology, we get
\begin{equation}\label{eq:E1_for_EMss}
\rmE^{t,n-t}_1(B,W)=\rmH^n(\Gr_W^t B)
=(\overline{\rmH}^\bullet(Z)^{\otimes-t}\otimes\rmH^\bullet(Y))^{n-t}
\end{equation}
from~\eqref{eq:E_1}. The~$\rmd_1$ differential comes from the bar resolution, which can be used to compute Tor. By~\eqref{eq:E_2}, this gives
\[
\rmE^{t,n-t}_2(B,W)=\Tor^{\rmH^\bullet(Z)}_{-t}(K,\rmH^\bullet(Y))^{n-t},
\]
and
\[
\rmH^n(B)=\rmH^n(\rmB(K,\rmC(Z),\rmC(Y)))=\rmH^n(\rmC(X))=\rmH^n(X)
\]
is the abutment.
\end{proof}

%%%

\subsection{Prelude to the Eilenberg--Moore spectral sequence}\label{sec:preEM}

We now consider the spectral sequence~\eqref{eq:t_spec_seq} for the particular situation~\eqref{eq:XYZ}, where we have a principal fibration sequence~\hbox{$\Omega Z\to X\to Y\to Z$} and the cochain complex~\hbox{$B=\rmB(K,\rmC(Z),\rmC(Y))\simeq\rmC(X)$} from the bar construction. The~$F$--filtration is given by a chosen cellular structure on~$Y$, and the~$W$--filtration is given by the bar filtration~\eqref{eq:W_filtration}. The spectral sequence~\eqref{eq:t_spec_seq} comes from the filtration on the associated graded cochain complex~$\Gr_W^tB$ that is induced by the~$F$--filtration. The main result here is the following.

\begin{theorem}\label{thm:LS_is_short}
For the complex~$B$ with the two filtrations~$F$ and~$W$ as above, the spectral sequence~\eqref{eq:t_spec_seq} always degenerates from its~$\rmE_2$ page on. It degenerates from its~$\rmE_1$ page on if and only if the space~$Y$ is~$K$--minimal~(see Definition~\ref{def:minimal}).
\end{theorem}

We start by describing the spectral sequence~\eqref{eq:t_spec_seq} in our situation.

\begin{proposition}\label{prop:description(3.4)}
For the cochain complex~$B$ with the two filtrations~$F$ and~$W$ as above, the spectral sequence~\eqref{eq:t_spec_seq} has
\[
\rmE^{s,t,u}_1=(\overline{\rmH}^\bullet(Z)^{\otimes-t}\otimes\rmH^\bullet(Y^s,Y^{s-1}))^q
\]
and
\[
\rmE_2^{s,t,u}=(\overline{\rmH}^\bullet(Z)^{\otimes-t}\otimes\rmH^s(Y))^q,
\]
with~$q=s+u=n-t$ as in Notation~\ref{not:pq}.
\end{proposition}

\begin{proof}
We start by noting that we look at the spectral sequence for the filtered complex
\begin{equation}\label{eq:Gr_W^tB}
\Gr_W^tB=\overline{\rmC}(Z)^{\otimes-t}\otimes\rmC(Y)[-t],
\end{equation}
with~$\Gr_W^tB^{t+u}=(\overline{\rmC}(Z)^{\otimes-t}\otimes\rmC(Y))^u$ and the shift as in~\eqref{eq:shift}. The filtration is induced by the~$F$--filtration, the cellular filtration on~$Y$. Therefore, we have 
\begin{equation}\label{eq:481}
\rmE_0^{s,t,u}=\Gr_F^s\Gr_W^tB^{s+t+u}=(\overline{\rmC}(Z)^{\otimes-t}\otimes\rmC(Y^s,Y^{s-1}))^{s+u},
\end{equation}
with~$t$ fixed throughout the spectral sequence. In other words, we can think of this as a family of spectral sequences, indexed by~$t$. The differential on the~$\rmE_0$ page is induced from~$B$. Since we have passed to the associated graded for the~$W$--filtration, we do not see the external component of the bar differential, but only the internal one, which the cochain complexes~$\rmC({?})$ bring with them. Passing to cohomology, we then get
\begin{equation}\label{eq:482}
\rmE^{s,t,u}_1=\rmH^n(\Gr_F^s\Gr_W^t B)=(\overline{\rmH}^\bullet(Z)^{\otimes-t}\otimes\rmH^\bullet(Y^s,Y^{s-1}))^{s+u}.
\end{equation}
The~$\rmd_1$ differential here is induced from the~$F$--filtration, in other words, it comes from the cells of~$Y$. The vector spaces~$\rmH^\bullet(Y^s,Y^{s-1})$, for varying~$s$, form the cellular cochain complex for the space~$Y$ with coefficients in the field~$K$. Because the functor~\hbox{${?}\longmapsto\overline{\rmH}^\bullet(Z)^{\otimes-t}\otimes{?}$} is exact, we get an isomorphism~\hbox{$\rmE^{s,t,u}_2\cong(\overline{\rmH}^\bullet(Z)^{\otimes-t}\otimes\rmH^s(Y))^{s+u}$}, as claimed.
\end{proof}

Using~$s+u=n-t$, the abutment of the spectral sequence~\eqref{eq:t_spec_seq} is
\[
\rmH^n(\Gr_W^tB)=\rmH^n(\overline{\rmC}(Z)^{\otimes-t}\otimes\rmC(Y)[-t])
=(\overline{\rmH}^\bullet(Z)^{\otimes-t}\otimes\rmH^\bullet(Y))^{n-t}
\]
by definition and the K\"unneth theorem. We have seen in~\eqref{eq:E1_for_EMss} that this abutment agrees with the~$\rmE_1$ page~$\rmE^{t,s+u}_1(B,W)=\rmE^{t,n-t}_1(B,W)$ of the Eilenberg--Moore spectral sequence, and also with
\[
\bigoplus_{s+u=n-t}\rmE_2^{s,t,u}=
\bigoplus_{s+u=n-t}(\overline{\rmH}^\bullet(Z)^{\otimes-t})^u\otimes\rmH^s(Y).
\]
As we shall see, the spectral sequence degenerates from its~$\rmE_2$ page on.

\begin{proposition}\label{prop:apply_exact_first}
For the complex~$B$ with the two filtrations~$F$ and~$W$ as above, the spectral sequence~\eqref{eq:t_spec_seq} is isomorphic, from its~$\rmE_1$ page on, to the spectral sequence obtained by applying the exact functor~${?}\mapsto\overline{\rmH}^\bullet(Z)^{\otimes-t}\otimes{?}$ to the spectral sequence constructed from the chosen cellular filtration~$F$ on the cochain complex~$\rmC(Y)$.
\end{proposition}

\begin{proof}
There is essentially only one reasonable way to prove this: we show that the two spectral sequences are induced by two equivalent cochain complexes with two equivalent filtrations. The spectral sequence~\eqref{eq:t_spec_seq} originates in the filtered complex~\eqref{eq:Gr_W^tB} with the~$F$--filtration from the chosen cellular structure on the space~$Y$. Since we are working over a field, we can choose an equivalence
\begin{equation}\label{eq:not_multiplicative}
\rmH^\bullet(Y)\overset{\sim}{\longrightarrow}\rmC(Y)
\end{equation}
of cochain complexes, where the left-hand side has zero differential. This induces an equivalence
\begin{equation}\label{eq:bar_equivalence}
\overline{\rmH}^\bullet(Z)^{\otimes-t}\otimes\rmC(Y)\overset{\sim}{\longrightarrow}\overline{\rmC}(Z)^{\otimes-t}\otimes\rmC(Y)
\end{equation}
of cochain complexes. The multiplicative structures which we have on both sides of~\eqref{eq:not_multiplicative} play no role because the tensor products are over the ground field~$K$. The equivalence~\eqref{eq:bar_equivalence} respects the~$F$--filtration, and, therefore, induces an isomorphism of spectral sequences from their~$\rmE_1$ pages on. The right-hand side of~\eqref{eq:bar_equivalence} gives the spectral sequence~\eqref{eq:t_spec_seq}. The left-hand side of~\eqref{eq:bar_equivalence} is obtained from the filtered cochain complex~$\rmC(Y)$, together with its cellular filtration~$F$, by applying the exact functor~\hbox{${?}\mapsto\overline{\rmH}^\bullet(Z)^{\otimes-t}\otimes{?}$}. The spectral sequence of the left-hand side is, therefore, obtained from that of the cochain complex~$\rmC(Y)$ with respect to the~$F$--filtration by applying the same functor, as claimed.
\end{proof}

\begin{proof}[Proof of Theorem~\ref{thm:LS_is_short}]
The spectral sequence for the cochain complex~$\rmC(Y)$ with respect to any cellular filtration~$F$ degenerates from its~$\rmE_2$ page on because the~$\rmE_1$ page is concentrated in its zeroth row. Applying any exact functor to it yields another spectral sequence that also degenerates from its~$\rmE_2$ page on. As we have seen in Proposition~\ref{prop:apply_exact_first}, when we apply the exact functor~\hbox{${?}\mapsto\overline{\rmH}^\bullet(Z)^{\otimes-t}\otimes{?}$}, we get the spectral sequence~\eqref{eq:t_spec_seq}. Therefore, this spectral sequence has to degenerate from its~$\rmE_2$ page on, too.

If the space~$Y$ in question is~$K$--minimal, then we have a cellular structure on it such that the spectral sequence for the cochain complex~$\rmC(Y)$ with respect to that cellular filtration~$F$ degenerates already from its~$\rmE_1$ page on, and the same argument as in the first part of the proof implies that the spectral sequence~\eqref{eq:t_spec_seq} does so, too.
\end{proof}

\begin{remark}
We can think of~\eqref{eq:t_spec_seq} as a cochain complex that computes the~$\rmE_1$ page of an Eilenberg--Moore spectral sequence. Note the difference to the~$\rmE_0$ term of the latter: there, we take the bar construction on the cochain complexes, whereas~\eqref{eq:t_spec_seq} involves only the cohomology of those.
\end{remark}

Let us spell out the consequence when the stronger hypothesis in Theorem~\ref{thm:LS_is_short} is satisfied:

\begin{corollary}\label{cor:factors}
When the space~$Y$ is~$K$--minimal, the Zassenhaus square factors the Eilenberg--Moore spectral sequence as a Leray--Serre spectral sequence after the spectral sequence~\eqref{eq:s_spec_seq}.
\end{corollary}

One consequence of Corollary~\ref{cor:factors} is that, for a~$K$--minimal space~$Y$, the Leray--Serre spectral sequence has at most as many non-trivial differentials as the Eilenberg--Moore spectral sequence: some terms might already have been killed by the bar differentials in~\eqref{eq:s_spec_seq}. The following Section~\ref{sec:preLS} contains our discussion of these preliminary bar differentials.

%%%

\subsection{Prelude to the Leray--Serre spectral sequence}\label{sec:preLS}

We now consider the spectral sequence~\eqref{eq:s_spec_seq}
for the particular situation~\eqref{eq:XYZ}, where we have a principal fibration sequence~\hbox{$\Omega Z\to X\to Y\to Z$} and the cochain complex~\hbox{$B=\rmB(K,\rmC(Z),\rmC(Y))\simeq\rmC(X)$} from the bar construction. The~$F$--filtration is given by a chosen cellular structure on~$Y$, and the~$W$--filtration is given by the bar filtration~\eqref{eq:W_filtration}. The spectral sequence~\eqref{eq:s_spec_seq} comes from the filtration on~$\Gr_F^sB$ that is induced by the~$W$--filtration. The main result here is the following.

\begin{theorem}\label{thm:EM_is_local}
For the complex~$B$ with the two filtrations~$F$ and~$W$ as above, the spectral sequence~\eqref{eq:s_spec_seq} degenerates from its~$\rmE_2$ page on for all~$Y$ if and only if this holds for~$Y\simeq\star$, that is if the  Eilenberg--Moore spectral sequence~\eqref{eq:EMSS_loops} for the loop space~$\Omega Z$ of~$Z$ degenerates from its~$\rmE_2$ page on. 
\end{theorem}

Recall from Definition~\ref{def:unbarred} that we have called those spaces~$Z$ that lead to degeneracy in Theorem~\ref{thm:EM_is_local}~$K$--unbarred. Therefore, we can refer to Example~\ref{ex:good_spaces_2} for examples of spaces~$Z$ where Theorem~\ref{thm:EM_is_local} applies~(see also Corollary~\ref{cor:E2} below for the resulting formula).

%%%

Recall our conventions for the cellular cochains on the space~$Y$ from Section~\ref{sec:cellular}: for all coefficients~$V$ we have set
\[
\Cell^s(Y;V)=\Hom(\rmH_\bullet(Y^s,Y^{s-1}),V)=\rmH^\bullet(Y^s,Y^{s-1})\otimes V,
\]
and this is exact as a functor in~$V$ because we are working over a field~$K$. We can use this to describe the spectral sequence~\eqref{eq:s_spec_seq} in our situation.

\begin{proposition}\label{prop:description(3.5)}
For the cochain complex~$B$ with the two filtrations~$F$ and~$W$ as above, the spectral sequence~\eqref{eq:s_spec_seq} has
\[
\rmE^{s,t,u}_1=\Cell^s(Y;(\overline{\rmH}^\bullet(Z)^{\otimes-t})^u)
\]
and
\[
\rmE_2^{s,t,u}=\Cell^s(Y;\Tor^{\rmH^\bullet(Z)}_{-t}(K,K)^u).
\]
Its abutment is~$\rmH^n(X^s,X^{s-1})$.
\end{proposition}

Note that, in this description, we recognize that the abutment of~\eqref{eq:s_spec_seq} is indeed the~$\rmE_1$ page~\eqref{eq:Leray--SerreE1} of the Leray--Serre spectral sequence for~$\rmH^\bullet(X)$ with respect to the chosen cellular structure on the space~$Y$.

\begin{proof}
We are looking at the spectral sequence
\[
\rmE^{s,t,u}_1=\rmH^n(\Gr_F^s\Gr_W^tB)
\Longrightarrow
\rmE^{s,t+u}_1(B,F)=\rmH^n(\Gr_F^sB).
\]
As in~\eqref{eq:481}, this spectral sequence starts with
\[
\rmE^{s,t,u}_0=\Gr_F^s\Gr_W^tB^{s+t+u}=(\overline{\rmC}(Z)^{\otimes-t}\otimes\rmC(Y^s,Y^{s-1}))^{s+u}.
\]
The~$s$ is always fixed. The~$\rmd_0$ differential is induced from~$B$, the totalization of the bar construction, and since we have passed to~$\Gr_W^t$, we do not see the external differential, but only the internal differential. Therefore, as in~\eqref{eq:482}, passing to cohomology, we get
\begin{equation}
\rmE^{s,t,u}_1
=\rmH^n(\Gr_F^s\Gr_W^tB)
=(\overline{\rmH}^\bullet(Z)^{\otimes-t}\otimes\rmH^\bullet(Y^s,Y^{s-1}))^{s+u}
\end{equation}
from the K\"unneth theorem. We note that~$\rmH^\bullet(Y^s,Y^{s-1})$, as an~$\rmH^\bullet(Z)$--module, is trivial and concentrated in degree~$s$, so that we can pull it out, and we find
\begin{equation}\label{eq:E1stu}
\rmE^{s,t,u}_1
=(\overline{\rmH}^\bullet(Z)^{\otimes-t})^u\otimes\rmH^\bullet(Y^s,Y^{s-1})
=\Cell^s(Y;(\overline{\rmH}^\bullet(Z)^{\otimes-t})^u).
\end{equation}
These identifications are compatible with the differential~$\rmd_1$ which is induced by the differential on the associated graded. This differential is the external part of the bar differential. This leads to
\[
\rmE^{s,t,u}_2
=\Cell^s(Y;\Tor^{\rmH^\bullet(Z)}_{-t}(K,K)^u)
\]
in the same way, as claimed. 

As for the abutment, we have~$\Gr_F^sB=\rmB(K,\rmC(Z),\rmC(Y^s,Y^{s-1}))\simeq\rmC(X^s,X^{s-1})$, where the equivalence is given by the relative version of the Eilenberg--Moore equivalence~(see our Proposition~\ref{prop:relativeEM} above). Therefore, passing to cohomology, we find that the abutment is indeed~\hbox{$\rmH^n(\Gr_F^sB)=\rmH^n(X^s,X^{s-1})$}.
\end{proof}

%%%

\begin{proposition}\label{prop:apply_exact}
For the cochain complex~$B$ with the two filtrations~$F$ and~$W$ as above, the spectral sequence~\eqref{eq:s_spec_seq} is isomorphic, from~$\rmE_1$ on, to the result of applying the exact functor~${?}\mapsto\Cell^s(Y;{?})$ to the Eilenberg--Moore spectral sequence
\[
\rmE_1^{t,u}=(\overline{\rmH}^\bullet(Z)^{\otimes-t})^u\Longrightarrow\rmH^{t+u}(\Omega Z)
\]
for the loop space~$\Omega Z$.
\end{proposition}

\begin{proof}
Again, there is essentially only one reasonable way to prove this: we show that the two spectral sequences are induced by two equivalent cochain complexes with two equivalent filtrations. And we already know from Proposition~\ref{prop:description(3.5)} that the~$\rmE_1$ pages are isomorphic: The spectral sequence~\eqref{eq:s_spec_seq} originates in the filtered complex 
\[
\Gr_F^sB=\Gr_F^s\rmB(K,\rmC(Z),\rmC(Y))\cong\rmB(K,\rmC(Z),\rmC(Y^s,Y^{s-1})),
\] 
with the~$W$--filtration from the bar construction. 

We can choose an equivalence
\begin{equation}\label{eq:eq}
\rmH^\bullet(Y^s,Y^{s-1})\overset{\sim}{\longrightarrow}\rmC(Y^s,Y^{s-1})
\end{equation}
from the zero-differential complex~$\rmH^\bullet(Y^s,Y^{s-1})$, concentrated in degree~$s$, to the cochain complex~$\rmC(Y^s,Y^{s-1})$. This is clear when working over the field~$K$, but we note that this does not just give an equivalence of~$K$--cochain complexes, but of~$\rmC(Y)$--cochain complexes, as~$\rmC(Y)$ acts trivially on the relative cochain complex~$\rmC(Y^s,Y^{s-1})$, i.e., via the ground field~$K$: if~$a\in\rmC^{\deg(a)}(Y)$ and~\hbox{$b\in\rmC^{\deg(b)}(Y^s,Y^{s-1})$} are two cochains, then their cup product in~$\rmC^{\deg(a)+\deg(b)}(Y^s,Y^{s-1})$ is defined by~\hbox{$(a\cup b)(y)=\pm a(y_{\deg(a)})\cdot b(y_{\deg(b)})$}, where~$y_{\deg(a)}$ is the~$\deg(a)$--front and~$y_{\deg(b)}$ is the~$\deg(b)$--back of the given simplex~$y$. If~$\deg(a)+\deg(b)=s$, we see that~$\deg(a)>0$ implies~$\deg(b)<s$, and then~$b(y_{\deg(b)})=0$ for~$b\in\rmC^{\deg(b)}(Y^s,Y^{s-1})$.

The equivalence~\eqref{eq:eq} of~$\rmC(Y)$--complexes induces an equivalence
\[
\rmB(K,\rmC(Z),\rmH^\bullet(Y^s,Y^{s-1}))\simeq\rmB(K,\rmC(Z),\rmC(Y^s,Y^{s-1})),
\]
which, by construction, respects the bar filtrations that we have on both sides. The right-hand side leads to the spectral sequence~\eqref{eq:s_spec_seq}. The left-hand side is isomorphic, again respecting the bar filtrations, to~$\Cell^s(Y;\rmB(K,\rmH^\bullet(Z),K))$. Since the latter is the result of applying the functor~${?}\mapsto\Cell^s(Y;{?})$ to the filtered cochain complex~$\rmB(K,\rmH^\bullet(Z),K)$, the result follows.
\end{proof}

\begin{remark}\label{rem:nil_2}
The preceding result implies an agreement
\[
\rmH^n(X^s,X^{s-1})\cong\Cell^s(Y;\rmH^{t+u}(\Omega Z))
\]
of the abutments of the two spectral sequences, too. Such a statement is a standard part of the identification of the~$\rmE_1$ page of the Leray--Serre spectral sequence, at least in the case when the fundamental group of the base acts trivially on the cohomology of the fibre. In our situation, the fundamental group of the space~$Y$ need not be trivial, but the action on~$\rmH^\bullet(\Omega Z)$ is, because it factors through the morphism~\hbox{$\pi_1(Y)=\pi_0(\Omega Y)\to\pi_0(\Omega Z)=\pi_1(Z)$}, and this is trivial if the space~$Z$ is simply-connected.
\end{remark}

\begin{proof}[Proof of Theorem~\ref{thm:EM_is_local}]
One direction is clear: if the spectral sequences degenerate for all spaces~$Y$, then, in particular, the ones for contractible spaces do so. 

For the other direction, assume that the Eilenberg--Moore spectral sequence for the loop space degenerates from its~$\rmE_2$ page on. Applying any exact functor to it yields another spectral sequence that also degenerates from its~$\rmE_2$ page on. As we have seen in Proposition~\ref{prop:apply_exact}, when we apply the exact functor~\hbox{${?}\mapsto\Cell^s(Y;{?})$}, we get the spectral sequence~\eqref{eq:s_spec_seq}. Therefore, this has to degenerate from its~$\rmE_2$ page on, too.
\end{proof}

\begin{corollary}\label{cor:E2}
If Theorem~\ref{thm:EM_is_local} applies, and the spectral sequence~\eqref{eq:s_spec_seq} does indeed degenerate from~$\rmE_2$ on, we get
\[
\rmH^n(X^s,X^{s-1})\cong
\bigoplus_{t+u=n-s}\Tor^{\rmH^\bullet(Z)}_{-t}(K,K)^u\otimes\rmH^\bullet(Y^s,Y^{s-1})
\]
for the~$\rmE_1$ page of the Leray--Serre spectral sequence. This holds whenever the space~$Z$ is~$K$--unbarred.
\end{corollary}

\begin{proof}
This follows immediately from what has been said before and the description of the~$\rmE_2$ term in Proposition~\ref{prop:description(3.5)}.
\end{proof}

We can summarize our findings by saying that, if we want to compute the cohomology~$\rmH^\bullet(X)$ of~$X$ by first running the spectral sequence~\eqref{eq:s_spec_seq} and then the Leray--Serre spectral sequence, then the cellular structure of~$Y$ does not interfere with the first spectral sequence~\eqref{eq:s_spec_seq} at all, only afterwards does it become relevant, in the Leray--Serre spectral sequence.

%%%

The following formulation applies our results to~$K$--minimal spaces.

\begin{theorem}
In our situation~\eqref{eq:XYZ}, if the space~$Y$ is~$K$--minimal, then there is a spectral sequence
\[
\rmE_2^{s,t,u}=\rmH^s(Y;\Tor_{-t}^{\rmH^\bullet(Z)}(K,K)^u)
\Longrightarrow
\rmH^s(Y;\rmH^{t+u}(\Omega Z))
\]
that computes the~$\rmE_2$ page of the Leray--Serre spectral sequence for the cohomology of~$X$. It degenerates from its~$\rmE_2$ page on if, in addition, the space~$Z$ is~$K$--unbarred.
\end{theorem}

\begin{proof}
When~$Y$ is~$K$--minimal, we have~$\Cell^s(Y;{?})\cong\rmH^s(Y;{?})$. Then, the spectral sequences in question are those from Proposition~\ref{prop:description(3.5)}, and the abutment is identified in Remark~\ref{rem:nil_2}. Theorem~\ref{thm:EM_is_local} gives the statement about the degeneration.
\end{proof}

%%%

\section{Examples}\label{sec:examples}

In this section, we briefly discuss the most critical test cases of the situation~\eqref{eq:XYZ} for our purposes. In particular, we discuss the spectral sequences when~\hbox{$X\simeq\star$}, or~\hbox{$Y\simeq\star$}, or~\hbox{$Z\simeq\star$}, respectively, and the three principal fibrations that feature the Hopf fibration~\hbox{$\eta\colon\rmS^3\to\rmS^2$}. These examples are chosen to be easy to work out; their purpose here is to supply evidence for the applicability of our theory. 

%%%

\subsection{Examples involving contractible spaces}

\begin{example}\label{ex:LS_diffs}
First, we consider the fibration sequences~$\Omega Z\to X\to Y\to Z$ where~$Y\simeq X\times Z$ and the map~$Y\to Z$ is equivalent to the projection. 
The Eilenberg--Moore spectral sequence
\[
\rmE^{p,q}_2=\Tor^{\rmH^\bullet(Z)}_{-p}(K,\rmH^\bullet(X\times Z))^q\Longrightarrow\rmH^{p+q}(X)
\]
always degenerates from~$\rmE_2$ on, because~$\rmH^\bullet(X\times Z)\cong\rmH^\bullet(X)\otimes\rmH^\bullet(Z)$ is a free~$\rmH^\bullet(Z)$--module, so that the~$\rmE_2$ page is just~$\rmH^\bullet(X)$ concentrated in the~$p=0$ column.
In contrast, the Leray--Serre spectral sequence can have arbitrarily long differentials, even in the simplest case of the situation, when the fibre~\hbox{$X\simeq\star$} is contractible so that the map~$Y\to Z$ is an equivalence. For a specific example, take~$Y=\rmS^r$ a sphere of dimension~$r\geqslant2$. Then, the Leray--Serre spectral sequence
\[
\rmE^{s,t}_2=\rmH^s(\rmS^r;\rmH^t(\Omega\rmS^r))\Longrightarrow\rmH^{s+t}(\star)
\]
has a nontrivial~$\rmd_r$. In contrast, the Eilenberg--Moore spectral sequence is concentrated in~$\rmE^{0,0}$, independent of the spaces~\hbox{$Y\simeq Z$}. The situation is even worse for spaces like~$\mathrm{BU}$.
\end{example}

\begin{example}\label{ex:EM_diffs}
We consider the fibration sequences~$\Omega Z\to X\to Y\to Z$ for spaces~$Y$ such that the spectral sequence
\[
\rmE^{s,t}_1=\rmH^{s+t}(Y^s,Y^{s-1})\Longrightarrow\rmH^{s+t}(Y),
\]
which is concentrated in the~$t=0$ row and which always degenerates from~$\rmE_2$ on, already degenerates from~$\rmE_1$ on. These are the spaces that we called~$K$--minimal~for the field~$K$ in Definition~\ref{def:minimal}. The results from Section~\ref{sec:preEM}, especially Theorem~\ref{thm:LS_is_short} and its Corollary~\ref{cor:factors}, apply. A fundamental example of the situation here is the case when the total space
$Y\simeq\star$
is contractible, so that the map~$\Omega Z\to X$ is an equivalence. 
In this case, the Leray--Serre spectral sequence
\[
\rmE^{s,t}_2=\rmH^s(\star;\rmH^t(\Omega Z))\Longrightarrow\rmH^{s+t}(\Omega Z)
\]
always degenerates from the beginning because it is concentrated in the~$s=0$ column. This even includes the~$\rmE_1$ page, because we can assume the cellular filtration to be trivial. 
In contrast, the Eilenberg--Moore spectral sequence 
\[
\rmE^{p,q}_2=\Tor^{\rmH^\bullet(Z)}_{-p}(K,K)^q\Longrightarrow\rmH^{p+q}(\Omega Z)
\]
can have arbitrarily long differentials. Specifically, take~$Z=\rmB\bbZ/\ell$ for an odd prime~$\ell$. Then, the Eilenberg--Moore spectral sequence
\[
\rmE^{p,q}_2=\Tor^{\rmH^\bullet(\rmB\bbZ/\ell)}_{-p}(\bbF_\ell,\bbF_\ell)^q\Longrightarrow\rmH^{p+q}(\bbZ/\ell)
\]
has a nonzero differential~$\rmd_{\ell-1}$~(see Baker--Richter~\cite[Sec.~6]{Baker+Richter} and compare Smith~\cite[Sec.~3]{Smith:IJM}). Eventually, the~$\ell$--dimensional~$\rmE_\infty$ page is spread out into~$1$--dimensional pieces for the Eilenberg--Moore spectral sequence, whereas it is concentrated in the~$\rmE^{0,0}_\infty$ spot for the Leray--Serre spectral sequence.
\end{example}

\begin{remark}
The preceding example contradicts the statement of~\cite[Ex.~8.12]{McCleary}: the Eilenberg--Moore spectral sequence does not have to degenerate at~$\rmE_2$, even if the Leray--Serre spectral sequence has all differentials coming from transgressions. Besides, we learned from Schochet~(private communication) that the attribution of that exercise to him is inaccurate.
\end{remark}

\begin{example}
We consider the fibration sequences~$\Omega Z\to X\to Y\to Z$ for spaces~$Z$ such that the Eilenberg--Moore spectral sequence
\[
\rmE^{p,q}_2=\Tor_{-p}^{\rmH^\bullet(Z)}(K,K)^q\Longrightarrow\rmH^{p+q}(\Omega Z)
\]
degenerates from its~$\rmE_2$ page on. These are the spaces that we called~$K$--unbarred in Definition~\ref{def:unbarred}. The results from Section~\ref{sec:preLS}, especially Theorem~\ref{thm:EM_is_local} and its Corollary~\ref{cor:E2}, apply. A special case of this situation is when the base space
$Z\simeq\star$
is contractible so that the map~$X\to Y$ is an equivalence. This is another example of product fibrations as in Example~\ref{ex:LS_diffs}. In this case, the Eilenberg--Moore spectral sequence has an~$\rmE_2$ page that is concentrated on the~$p=0$ column; it degenerates already from~$\rmE_1$ on. The Leray--Serre spectral sequence has an~$\rmE_2$ page that is concentrated on the~$t=0$ line; it degenerates from~$\rmE_2$ on, but on the~$\rmE_1$ page one might still see differentials from the cellular structure on~$Y$. As we have already mentioned in Example~\ref{ex:good_spaces_2}, more interesting instances of the situation here are given in the case when the base space~$Z$ is a torus~$(\rmS^1)^d$ or a sphere~$\rmS^d$ where~$K$ is any field, but also~$\bbR P^\infty\simeq\rmB\bbZ/2$ when the characteristic of the field~$K$ is~$2$.
\end{example}

%%%

\subsection{Principal fibrations involving the Hopf map}

\begin{example}\label{ex:not_dec}
Consider the fibration sequence
$
\rmS^3
\overset{\eta}{\longrightarrow}\rmS^2
\longrightarrow\bbC\rmP^\infty
\longrightarrow\bbH\rmP^\infty.
$
The projective spaces are~$K$--minimal and~$K$--unbarred, so that we are in the best possible situation. Since the cohomology ring~$\rmH^\bullet(\bbC\rmP^\infty)$ is free as a module over the algebra~$\rmH^\bullet(\bbH\rmP^\infty)$, on two generators in degree~$0$ and~$2$, the Eilenberg--Moore spectral sequence degenerates at its~$\rmE_2$ page to compute~$\rmH^\bullet(\rmS^2)$. On the other hand, the Leray--Serre spectral sequence needs infinitely many~$\rmd_4$'s to achieve the same goal. Thus, there is no way that a reindexing of the pages using Deligne's {\it d\'ecalage} \cite[1.3.3, 1.3.4]{Deligne:II} could explain the relationship between the two spectral sequences as a d\'ecalage effect.

\end{example}

\begin{example}\label{ex:E3deg}
Consider the fibration sequence
$
\rmS^1\longrightarrow\rmS^3
\overset{\eta}{\longrightarrow}\rmS^2
\longrightarrow\bbC\rmP^\infty.
$
The spheres and the projective space are all~$K$--minimal and~$K$--unbarred so that we are in the best possible situation. When computing~$\rmH^\bullet(\rmS^3)$, the Eilenberg--Moore spectral sequence degenerates from its~$\rmE_2$ page on, but the Leray--Serre spectral sequence needs a nontrivial~$\rmd_2$ differential before it degenerates from its~$\rmE_3$ page on.
\end{example}

\begin{example}
Consider the fibration sequence
$
\Omega\rmS^2\longrightarrow
\rmS^1\longrightarrow\rmS^3
\overset{\eta}{\longrightarrow}\rmS^2.
$
The spheres are all~$K$--minimal and~$K$--unbarred, so that we are in the best possible situation. We look at the two spectral sequences that compute~$\rmH^\bullet(\rmS^1)$. We see that the Eilenberg--Moore spectral sequence has non-trivial~$\rmd_2$ differentials, and the Leray--Serre spectral sequence has non-trivial~$\rmd_3$ differentials. The theory in Section~\ref{sec:comparison} applies. Both preludes degenerate and show that the Eilenberg--Moore spectral sequence and the Leray--Serre spectral sequence are bigraded incarnations of the same tri-graded spectral sequence. In fact, we have
\[
\rmE^{s,t,u}_1
=\Cell^s(\rmS^3;(\overline{\rmH}^\bullet(\rmS^2)^{\otimes-t})^u)
=
\begin{cases}
K & s\in\{0,3\}\text{ and }t\geqslant0\text{ and }u=2t\\
0 & \text{otherwise}.
\end{cases}
\]
from~\eqref{eq:E1stu}. Under the projections~\eqref{eq:W_iso} and~\eqref{eq:F_iso}, these give rise to the usual~$\rmE_2$ pages, and the index transformation~\eqref{eq:as_in_decalage} shifts the Eilenberg--Moore~$\rmd_2$ differential to a Leray--Serre~$\rmd_3$.
\end{example}

%%%

\section*{Acknowledgement}

We thank the referee for their suggestions to improve the exposition.

%%%

\parskip5.5pt

%%%

\vfill
\parbox{\linewidth}{%
%Frank Neumann,
Dipartimento di Matematica `Felice Casorati',
Universit\`a di Pavia,
via Ferrata, 5, 
27100 Pavia,
ITALY\\ %\phantom{ }\\
\href{mailto:frank.neumann@unipv.it}{frank.neumann@unipv.it}}

\vspace{\baselineskip}

\parbox{\linewidth}{%
%Markus Szymik,
Department of Mathematical Sciences,
NTNU Norwegian University of Science and Technology,
7491 Trondheim,
NORWAY\\ %\phantom{ }\\
\href{mailto:markus.szymik@ntnu.no}{markus.szymik@ntnu.no}}

\vspace{.1\baselineskip}

\parbox{\linewidth}{%
School of Mathematics and Statistics,
The University of Sheffield,
Sheffield S3 7RH,
UNITED KINGDOM\\ %\phantom{ }\\
\href{mailto:m.szymik@sheffield.ac.uk}{m.szymik@sheffield.ac.uk}}
\end{document}